\documentclass[twocolumn,letterpaper]{IEEEAerospaceCLS}  

\usepackage[]{graphicx}    
\newcommand{\ignore}[1]{}  

\usepackage{cite}
\usepackage{amsmath,amssymb,amsfonts,amsthm}
\usepackage{algorithmic}
\usepackage{textcomp}
\usepackage{xcolor}
\usepackage{optidef}
\usepackage{balance}

\newcommand{\lp}{ \left( }
\newcommand{\rp}{ \right) }

\newcommand{\xk}{ x_k }
\newcommand{\yk}{ y_k }
\newcommand{\uk}{ u_k }

\newcommand{\xkp}{ x_{k+1} }

\newcommand{\Rn}{ {\mathbb{R}}^n }
\newcommand{\Rm}{ {\mathbb{R}}^m }
\newcommand{\Rnn}{ {\mathbb{R}}^{n \times n}}
\newcommand{\Rnm}{ {\mathbb{R}}^{n \times m}}
\newcommand{\Rmm}{ {\mathbb{R}}^{m \times m}}

\newtheorem{theorem}{Theorem}[section]
\newtheorem{lemma}[theorem]{Lemma}

\begin{document}
\title{A Model-Free Data-Driven Algorithm for Continuous-Time Control}

\author{%
Sean R. Bowerfind\\ 
Department of ECE\\
Auburn University\\
Auburn, AL 36849\\
srb0063@auburn.edu
\and 
Matthew R. Kirchner\\
Department of ECE\\
Auburn University\\
Auburn, AL 36849\\
kirchner@auburn.edu
\and
Gary A. Hewer\\
Physics and Computational Sciences\\
Naval Air Warfare Center Weapons Division\\
China Lake, CA \\
gary.a.hewer.civ@us.navy.mil
\and
D. Reed Robinson\\
Autonomy and Signal Processing\\
Naval Air Warfare Center Weapons Division\\
China Lake, CA \\
david.r.robinson11.civ@us.navy.mil
\and
Paula Chen\\
Autonomy and Signal Processing\\
Naval Air Warfare Center Weapons Division\\
China Lake, CA \\
paula.x.chen.civ@us.navy.mil
\and
Alireza Farahmandi\\
Autonomy and Signal Processing\\
Naval Air Warfare Center Weapons Division\\
China Lake, CA \\
alireza.farahmandi.civ@us.navy.mil
\and
Katia Estabridis\\
Physics and Computational Sciences\\
Naval Air Warfare Center Weapons Division\\
China Lake, CA \\
katia.n.estabridis.civ@us.navy.mil
\thanks{\footnotesize 979-8-3503-5597-0/25/$\$31.00$ \copyright2025 IEEE}              
}

\maketitle

\thispagestyle{plain}
\pagestyle{plain}

\maketitle

\thispagestyle{plain}
\pagestyle{plain}

\begin{abstract}
Presented is an algorithm to synthesize an infinite-horizon LQR optimal feedback controller for continuous-time systems. The algorithm does not require knowledge of the system dynamics, but instead uses only a finite-length sampling of (possibly suboptimal) input-output data. The algorithm is based on a constrained optimization problem that enforces a necessary condition on the dynamics of the optimal value function along an arbitrary trajectory. This paper presents the derivation as well as shows examples applied to both linear and nonlinear systems inspired  by air vehicles.
\end{abstract}

\tableofcontents

\section{Introduction}\label{sec:introduction}
The linear quadratic regulator (LQR) has been one of the more useful state-space control design methodologies as it provides many beneficial properties, including simple hardware implementation and guarantees with respect to both phase and gain margin \cite{safonov1977gain,lehtomaki1981robustness}. Classical LQR design assumes linear dynamics and relies on complete knowledge of system matrices, $A$ and $B$, to compute the optimal feedback gain matrix. These matrices typically come from either direct knowledge of an inherently linear system or through a linearization of a known nonlinear system. 

For many systems, however, direct knowledge of the system is not known a-priori, but the collection of data is readily available. When the system is thought to be linear, it is possible to attempt model identification techniques to find the $A$ and $B$ matrices. In this case, a model-free approach can be seen to streamline the design procedure by going directly from data to LQR controller, skipping the model identification.

A more common situation arises when the system is nonlinear and an LQR controller is desired from a linearized approximation in a neighborhood around an equilibrium point. In this case, model identification of the nonlinear model directly, in order to later linearize, is more difficult. Therefore, a model-free approach should be practical in that data collected around an equilibrium point of a nonlinear system should give rise to an LQR controller that is similar to the controller generated by a model created by linearization. Additionally, any data-driven method must be robust to noise in the data, as no data can be collected from a real system without noise in the measurements. 

There have been some recent attempts to compute LQR directly from data \cite{vrabie_optimal_2013}, and typically are formulated for discrete-time systems. Most previous works use an online and so-called on-policy\footnote{For a discussion between on-policy and off-policy methods, see \cite[Sec. 2.1, p. 145]{luo2014off}.} methodology and that requires actively updating the new control policy to generate the next data sample. This may require the entire adaptive algorithm to be implemented on an onboard embedded processor, which is problematic for platforms that have limited processing resources or restrictions on editing the platform's firmware. Alternatively, an online method could be incorporated into a computer model or simulation of the desired platform. However, if the user has enough knowledge of the system to construct a simulation, then the user has enough system knowledge for traditional design techniques and model-free approaches are unnecessary.

An example of one such online method, distinctive in that it considers continuous-time dynamics, is presented in \cite{vamvoudakis_q-learning_2017}. An adaptive control technique is developed in this work to update candidate LQR control solutions which require careful tuning of hyperparameters, since otherwise it can be susceptible numerical problems.

Offline approaches have been investigated. Most notable is the work of \cite{farjadnasab2022model}. This work assumes discrete-time dynamics and uses this representation to remove the dependence on $A$. This formulation can induce a sensitivity to noise in the data, which will be more prominent when the sample time in the data becomes small. 

In this paper we present a practical model-free approach to synthesize an LQR controller directly from data. We assume the data is collected in the form of noisy observations of the state\footnote{Or observation of a function of the state.} and the control sequence that drives the system. Since the optimal cost-to-go for LQR has a known parametric form, we can evaluate the cost-to-go along any arbitrary trajectory. We show this cost-to-go must satisfy a specific dynamic constraint and is a necessary condition for the value function. We then construct a nonlinear programming problem (NLP) whereby the above dynamic constraint is enforced along the observed data and we minimize the observation error.

We demonstrate the proposed method on two examples. The first is where data is generated from a benchmark linear system, inspired by a Boeing 747. This example is chosen since we can compare the gain matrix computed from the model free method to the ground truth using traditional methods where $A$ and $B$ are known. The second example generates data from a nonlinear quadcopter by perturbing the system manually around the hover trim condition.   

\section{Preliminaries on LQR Optimal Control}\label{sec:LQR prelims}
Consider the following linear time-invariant continuous time dynamical system
\begin{equation}
    \begin{cases}
        \dot{x}(t) = f\lp x,u\rp :=  Ax(t) + Bu(t), \\
        x\lp 0\rp = x_0,
    \end{cases}
    \label{eq:LTIsys}
\end{equation}
for $t\in \left[0,\infty\right)$ where $x(t) \in \Rn$ is the system state and $u(t) \in \Rm$ is the control input. We assume without loss of generality that the initial time is always $t_0=0$ since the system \eqref{eq:LTIsys} is time invariant. As a shorthand, we denote the true state trajectory by $x\lp t\rp$: 
$\left[0,\infty \right)\ni t\mapsto x\lp t;x_0,u\lp\cdot\rp\rp\in\Rn$. When $t$ appears, we assume it is the above definition that evolves in time with measurable control sequence $u\lp\cdot\rp\in\Rm$,
according to \eqref{eq:LTIsys} starting from initial state $x_{0}$ at $t=0$. When $x$ appears without dependence on $t$, we are referring to a specific state without regard to an entire trajectory. 

We consider an infinite time horizon cost functional with an LQR running cost, $\ell\lp x,u\rp$, given by 
\begin{align}
J\lp x_{0},u\lp\cdot\rp\rp &=\int_{0}^{\infty} \ell\left(x\left(t\right),u\left(t\right)\right)dt \notag \\
 &=\int_{0}^{\infty} x\lp t\rp^{\top}Mx\lp t\rp+u\lp t\rp^{\top}Ru\lp t\rp dt,\label{eq:infinite cost funct}
\end{align}

with $x(0)=x_0$ and where both $\Rnn\ni M=M^\top \succeq 0$ and $\Rmm \ni R=R^\top \succ 0$ are the user-supplied weight matrices.

The value function $V:\Rn \rightarrow {\mathbb{R}}$ is defined as the minimum cost, $J$, among all admissible controls for an initial state $x_0$ given as
\begin{equation}
     V\lp x_0 \rp = \inf_{u\lp\cdot\rp\in\Rm}  J\lp x_0, u\lp(\cdot\rp \rp,
     \label{eq:value funct}
\end{equation}
subject to the dynamic system constraint \eqref{eq:LTIsys}. Finding a feedback control policy that minimizes \eqref{eq:value funct} is known widely as the infinite horizon linear quadratic regulator (LQR) problem. It is well known that the value function takes the form
\begin{equation}
    V(x) = x^\top P x,
    \label{eq:optimal value funct}
\end{equation}
where $\Rnn \ni P=P^\top \succeq 0$. Without loss of generality, we assume
\[
M=C^{\top}C,
\]
since $M$ is symmetric positive semi-definite \cite[Thm. 7.2.7, p. 440]{horn2012matrix}, but also could be seen
as constructing a cost functional \eqref{eq:infinite cost funct} in terms of a controlled output $z=Cx$ \cite[Chap. 20, p. 191]{hespanha2018linear}. It is well known \cite{kailath1980linear} that when the pair $\lp A,B\rp$
is stabilizable and the pair $\lp A,C\rp$ is detectable, then
$P$ is the unique positive semi-definite solution of the algebraic
Riccati (ARE) equation
\begin{equation}
    A^\top P + PA - PBR^{-1}B^\top P + M = 0,
    \label{eq:ARE}
\end{equation}
and the optimal closed loop control is
\begin{equation}
u^{*}=-Kx,\label{eq:optimal control}
\end{equation}
where
\begin{equation}
    K=R^{-1}B^{\top}P.\label{eq:LQR optimal K}
\end{equation}

\section{Implicit Model-Free Formulation}\label{sec:model-free formulation}
For any $P$ satisfying \eqref{eq:ARE}, we can evaluate the candidate value function,
\eqref{eq:optimal value funct}, along any (not necessarily optimal) trajectory of the system state. This is the scalar function $\left[0,\infty\right)\ni t\mapsto V\lp x\lp t\rp\rp\in{\mathbb{R}}$
and must satisfy along the trajectory $\left[0,\infty\right)\ni t\mapsto x\lp t\rp$:
\begin{align}
\dot{V}\lp x\lp t\rp\rp & =\frac{\partial V}{\partial x}\dot{x}\lp t\rp\nonumber \\
 & =\frac{\partial V}{\partial x}\lp Ax\lp t\rp+Bu\lp t\rp\rp\nonumber \\
 & =2x\lp t\rp^{\top}PAx\lp t\rp+2x\lp t\rp^{\top}PBu\lp t\rp,\label{eq:implicit value function}
\end{align}
where the last line is from
\[
\frac{\partial}{\partial x}\left\{ x^{\top}Px\right\} =2x^{\top}P.
\]
Under a set of mild Lipschitz continuity assumptions, there exists
a unique value function \eqref{eq:value funct} that satisfies the following stationary
(time-independent) Hamilton--Jacobi partial differential equation
(PDE)
\begin{equation}
\underset{u\in\Rm}{\min}H\lp x,u,\frac{\partial V}{\partial x}\rp=0,\label{eq:HJB equation for inf LQR}
\end{equation}
where
\[
H\lp x,u,p\rp:=\left\langle p,f\lp x,u\rp \right\rangle+\ell\lp x,u\rp.
\]
This implies that for the system given in \eqref{eq:LTIsys} and the cost functional
given in \eqref{eq:infinite cost funct} we have
\begin{align}
H\lp x,u,\frac{\partial V}{\partial x}\rp = &\frac{\partial V}{\partial x}\lp Ax+Bu\rp+x^{\top}Mx+u^{\top}Ru\nonumber \\
 =&2x^{\top}PAx+2x{}^{\top}PBu\nonumber\\
 &+x^{\top}Mx+u^{\top}Ru.\label{eq:raw hamiltonian}
\end{align}
We can determine the extremal control in \eqref{eq:HJB equation for inf LQR} by taking the gradient of \eqref{eq:raw hamiltonian}
with respect $u$, setting to zero, and solving. This results in
the following optimal control
\begin{equation}
u^{*}=-R^{-1}B^{\top}Px.\label{eq:optimal control hamiltonian}
\end{equation}
This implies that, after substituting
\eqref{eq:optimal control hamiltonian} into \eqref{eq:raw hamiltonian}, the equality in \eqref{eq:HJB equation for inf LQR} can be written as
\[
2x^{\top}PAx-x^{\top}PBR^{-1}B^{\top}Px+x^{\top}Mx=0,
\]
which gives the expression
\begin{equation}
2x^{\top}PAx=x^{\top}PBR^{-1}B^{\top}Px-x^{\top}Mx.\label{eq:Sub A out}
\end{equation}
When \eqref{eq:Sub A out} is substituted in \eqref{eq:implicit value function},
we remove our dependence on the knowledge of system matrix $A$,
and we have
\begin{align}
\dot{V}\lp x\lp t\rp\rp & =x\lp t\rp^{\top}PBR^{-1}B^{\top}Px\lp t\rp\nonumber \\
 & -x\lp t\rp^{\top}Mx\lp t\rp+2x\lp t\rp^{\top}PBu\lp t\rp.\label{eq:cont. time equality constraint}
\end{align}
Therefore, for any \emph{arbitrary} control sequence $u\left(t\right)$ and corresponding state trajectory $x\left(t\right)$,
the dynamic equality \eqref{eq:cont. time equality constraint} must hold for optimal $P$.

\section{Data-Driven Control Synthesis}\label{sec:data driven control}
We consider a trajectory segment of finite length on $\left[0,T\right]$
for some $0<T<\infty$, whereby $N$ samples are observed on the time
grid
\[
\pi^{N}=\left\{ t_{k}:k=0,\ldots N\right\} .
\]
The samples observed are the control inputs, $u\lp t_{k}\rp$,
and output $y\lp t_{k}\rp$, where we do not assume direct observation
of the full state, but rather a classical linear observation model perturbed
by noise
\[
y\lp t_{k}\rp =Cx\lp t_{k}\rp+\epsilon,
\]
where the noise follows a zero mean Gaussian distribution
\[
\epsilon\sim{\mathcal{N}}\lp 0,\Sigma\rp,
\]
with $\Sigma$ being the covariance matrix. We will use the shorthand
\begin{equation}\label{eq:notation}
\begin{cases}
\xk:=x\lp t_{k}\rp,\\
\yk:=y\lp t_{k}\rp,\\
\uk:=u\lp t_{k}\rp,\\
V_{k}:=V\lp x\lp t_{k}\rp\rp,
\end{cases}
\end{equation}
for any $k\in\pi^{N}$. Additionally, we denote $X$ as the collection
of all states defined by
\[
X:=\left[x_{0},\ldots,x_{N}\right],
\]
and we denote $Y$ as the collection of all output observations
defined by
\[
Y:=\left[y_{0},\ldots,y_{N}\right].
\]
We discretize the rate in \eqref{eq:cont. time equality constraint}
so that for any time $t_{k}$
\[
\dot{V}\lp x\lp t_{k}\rp\rp\approx D_{k}^{V},
\]
where we designate $D_{k}^{V}$ as a finite difference approximation
to $\dot{V}\lp x\lp t_{k}\rp \rp$. We will not require
any specific finite differencing form or method, so for simplicity
we use the first-order forward Euler
\begin{equation}
D_{k}^{V}=\frac{V_{k+1}-V_{k}}{\Delta t},\label{eq:finte diff in V}
\end{equation}
with $\pi^{N}$ consisting of uniformly spaces time samples\footnote{Recall that we define $N$ as the number of samples \emph{not} including
index 0.}
\[
\Delta t=\frac{T}{N}.
\]
Many approximations exist for $D_{k}^{V}$, for example higher-order
finite differencing schemes \cite{mathews2004numerical} as well as trajectory representations
for both uniformly spaced \cite{cichella2018bernstein} and non-uniform time grids \cite{elnagar1995pseudospectral}.
These can increase computational accuracy and can be applied to what
follows but are outside the scope of this work.

Proceeding with $D_{k}^{V}$ as in \eqref{eq:finte diff in V}, we write this as
\[
D_{k}^{V}=\frac{1}{\Delta t}\lp x_{k+1}^{\top}Px_{k+1}-x_{k}^{\top}Px_{k}\rp.
\]
Defining $S:=PB$, the equality \eqref{eq:cont. time equality constraint}
implies the following holds for all $k=0,\ldots N-1$: 
\begin{align}
\frac{1}{\Delta t}\lp x_{k+1}^{\top}Px_{k+1}-x_{k}^{\top}Px_{k}\rp & =x_{k}^{\top}SR^{-1}S^{\top}x_{k}\nonumber \\
 & -x_{k}^{\top}Mx_{k}+2x{}_{k}^{\top}Su_{k}.\label{eq:discrete equality}
\end{align}
Given the set of observations $\left\{ \uk,\yk\right\} _{k=0,\ldots,N}$
we seek to find the pair $\lp P,S\rp$, with $P$ being symmetric
positive semi-definite, that satisfies \eqref{eq:discrete equality}.
This is achieved by the following constrained optimization:
\begin{equation}
\begin{cases}
\underset{L,S,X}{\min} & \left\Vert \text{vec}\left(Y\right)-\text{vec}\left(CX\right)\right\Vert _{2}^{2}\\
\\
\underset{\forall k=0,\ldots,N-1}{\text{subject to}} & \frac{1}{\Delta t}\lp\xkp^{\top}L^{\top}L\xkp-\xk^{\top}L^{\top}L\xk\rp\\
 & -\xk^{\top}SR^{-1}S^{\top}\xk\\
 & +\xk^{\top}M\xk-2\xk^{\top}S\uk=0.
\end{cases}\label{eq:NLP}
\end{equation}
where $L$ is defined as $P=L^{\top}L$, and where $\text{vec}\lp\cdot\rp$
is the vectorize operator that reshapes a matrix into a column vector.
We solve for $L$ instead of $P$ directly as this enforces the constraint
that $P$ must be symmetric positive semi-definite \cite[Thm. 7.2.7, p. 440]{horn2012matrix}. The LQR
optimal feedback control follows from \eqref{eq:NLP} with $u=-Kx$ where
\[
K=R^{-1}S^{\top}.
\]

\subsection{A Note On Data Collection}

 Control inputs can be recorded from either an open-loop or a closed-loop process. The latter represents a case where a rudimentary controller is used to stabilize the system to enable data collection on an otherwise unstable system. Additionally, the model-free controller can be used to refine a suboptimal closed-loop controller based on real-world test data.
 
 As is the case with any method based solely on sampled data, care must be taken with respect to how the data is collected\footnote{As a concrete example, see the model-identification method known as Dynamic Mode Decomposition\cite{proctor2016dynamic}, and the corresponding note in \cite[p. 395]{brunton2022data}.}. Qualitatively, control inputs should produce a trajectory that adequately spans the $n$-dimensional state space of the original open-loop system. Typically, an assumption of
 \emph{persistent excitation}\cite{willems2005note,van2020willems} with respect to the input is made. Formal quantification of the diversity of the data, while an active area of current research, is beyond the scope of this current article and is left to future work.

\section{Connection To Q-Learning}\label{sec:Q-learning}
Many existing methods attempting data-driven approaches to find
the controller without knowledge of the system matrices, $A$ and
$B$, do so under the so-called reinforcement learning paradigm. We
will show here that the method presented above can be equivalently
seen as a variant of this framework. We note that while the implicit
dynamic constraints given in \eqref{eq:cont. time equality constraint} are mathematically equivalent in the
continuous setting to what follows, it is observed that \eqref{eq:discrete equality} has significantly
greater numeric stability when implemented on sampled-data.

We begin by taking a Q-learning \cite{watkins1989learning,watkins1992q} point of view to
the above problem, and define a Q-function for a continuous-time case \cite{luo2014q} with
\[
Q\lp t;x \lp \cdot \rp,u\lp\cdot\rp\rp:=\int_{t}^{\infty}x\lp\tau\rp^{\top}Mx\lp\tau\rp+u\lp\tau\rp^{\top}Ru\lp\tau\rp d\tau.
\]
Note that the Q-function is simply an evaluation of the cost functional,
$J$ from \eqref{eq:infinite cost funct}, evaluated on an arbitrary trajectory, starting at time
$t$. This implies the following semi-group property along a trajectory
\begin{align}
Q\lp t;x\lp \cdot\rp,u\lp\cdot\rp\rp = & Q\lp t+T;x\lp \cdot \rp,u\lp\cdot\rp\rp\nonumber \\
 +&\int_{t}^{t+T}x\lp\tau\rp^{\top}Mx\lp\tau\rp+u\lp\tau\rp^{\top}Ru\lp\tau\rp.\label{eq:Q semi-group}
\end{align}
Trivially, the optimal $Q$ function is equivalent to the value function,
i.e.
\[
Q^{*}\lp t\rp :=\underset{u\lp\cdot\rp}{\min}\,Q\lp t;x\lp \cdot \rp,u\lp\cdot\rp\rp=V\lp x\lp t \rp \rp.
\]
Proceeding with an \emph{advantage learning} framework \cite{baird_reinforcement_1994}, we
define the advantage function as
\begin{equation}
{\mathbb{A}}\lp t;x \lp \cdot \rp ,u\lp\cdot\rp\rp:=Q\lp t;x(\cdot),u\lp\cdot\rp\rp-V\lp x(t)\rp.\label{eq:Advantage def}
\end{equation}

The following lemma introduces one such advantage function that will be useful in our analysis.
\begin{lemma}\label{lem:advantage}
Suppose $P$ satisfies the algebraic Riccati equation. Then the following advantage function satisfies \eqref{eq:Q semi-group} $:$
\begin{equation}
{\mathbb{A}}\lp t;x \lp \cdot \rp,u\lp\cdot\rp\rp=\int_{t}^{\infty}\left\Vert u\lp\tau\rp+R^{-1}B^{\top}Px\lp\tau\rp\right\Vert _{R}^{2}d\tau,\label{eq:Advantage formula}
\end{equation}
where $\|\cdot\|_{R}=\sqrt{\langle\cdot,R\cdot\rangle}$ denotes the
norm induced by the symmetric positive definite matrix $R$.
\end{lemma}

The proof is given in the appendix.

Lemma \ref{lem:advantage} allows us to write $\eqref{eq:Q semi-group}$ as
\begin{align*}
 & V\lp x\lp t\rp\rp+\int_{t}^{t+T}\left\Vert u\lp\tau\rp+R^{-1}B^{\top}Px\lp\tau\rp\right\Vert _{R}^{2}d\tau\\
 & =V\lp x\lp t+T\rp\rp\\
 & +\int_{t}^{t+T}x\lp\tau\rp^{\top}Mx\lp\tau\rp+u\lp\tau\rp^{\top}Ru\lp\tau\rp d\tau.
\end{align*}
Referencing \eqref{eq:optimal value funct} and using the notation defined in \eqref{eq:notation} we arrive at the following equality:
\begin{align*}
 & \lp x_{k+1}^{\top}Px_{k+1}-x_{k}^{\top}Px_{k}\rp\\
 =&\int_{t}^{t+T}\left\Vert u\lp\tau\rp+R^{-1}B^{\top}Px\lp\tau\rp\right\Vert _{R}^{2}d\tau\\
  &-\int_{t}^{t+T}x\lp\tau\rp^{\top}Mx\lp\tau\rp+u\lp\tau\rp^{\top}Ru\lp\tau\rp d\tau.
\end{align*}
If Euler's method is used to discretize the integral
terms above, then after some algebra we arrive at the same expression as \eqref{eq:discrete equality}.

\section{Results}\label{sec:results}
Two results cases will now be presented. First, the solution methodology is verified using a known linear system. Data is generated from the linear system and a controller is found by solving \eqref{eq:NLP}. We refer to the state feedback gain matrix computed by the proposed model-free method as $K_{\text{MF}}$. We then compare the model-free gain matrix to the one computed when $A$ and $B$ are known using a classical ARE approach. We denote this gain matrix as $K_\text{LQR}$. This comparison is illustrated in Figure \ref{fig:flowchart}.

Second, the model-free controller methods are applied to stabilize an unknown nonlinear system about an equilibrium point. This example demonstrates the versatility of the approach, especially on systems that are unknown, difficult to linearize, or have high dimensional state spaces. The gain matrix from the model-free approach is compared to the gain matrix using $A$ and $B$ generated by linearizing the known nonlinear model about the equilibrium point.

All simulation results have been implemented using MATLAB R2024a and Simulink.  For the optimization in \eqref{eq:NLP} we use the NLP code IPOPT \cite{wachter2006implementation}, where
the constraint Jacobian was computed using automatic differentiation by the methods of \cite{andersson2019casadi}. The NLP was initialized by setting $X = Y $, $L = I \in \Rnn$, and $S = {\mathbf{1}} \in \Rnm$.

\begin{figure}
    \centering
    \includegraphics[width=1.0\linewidth]{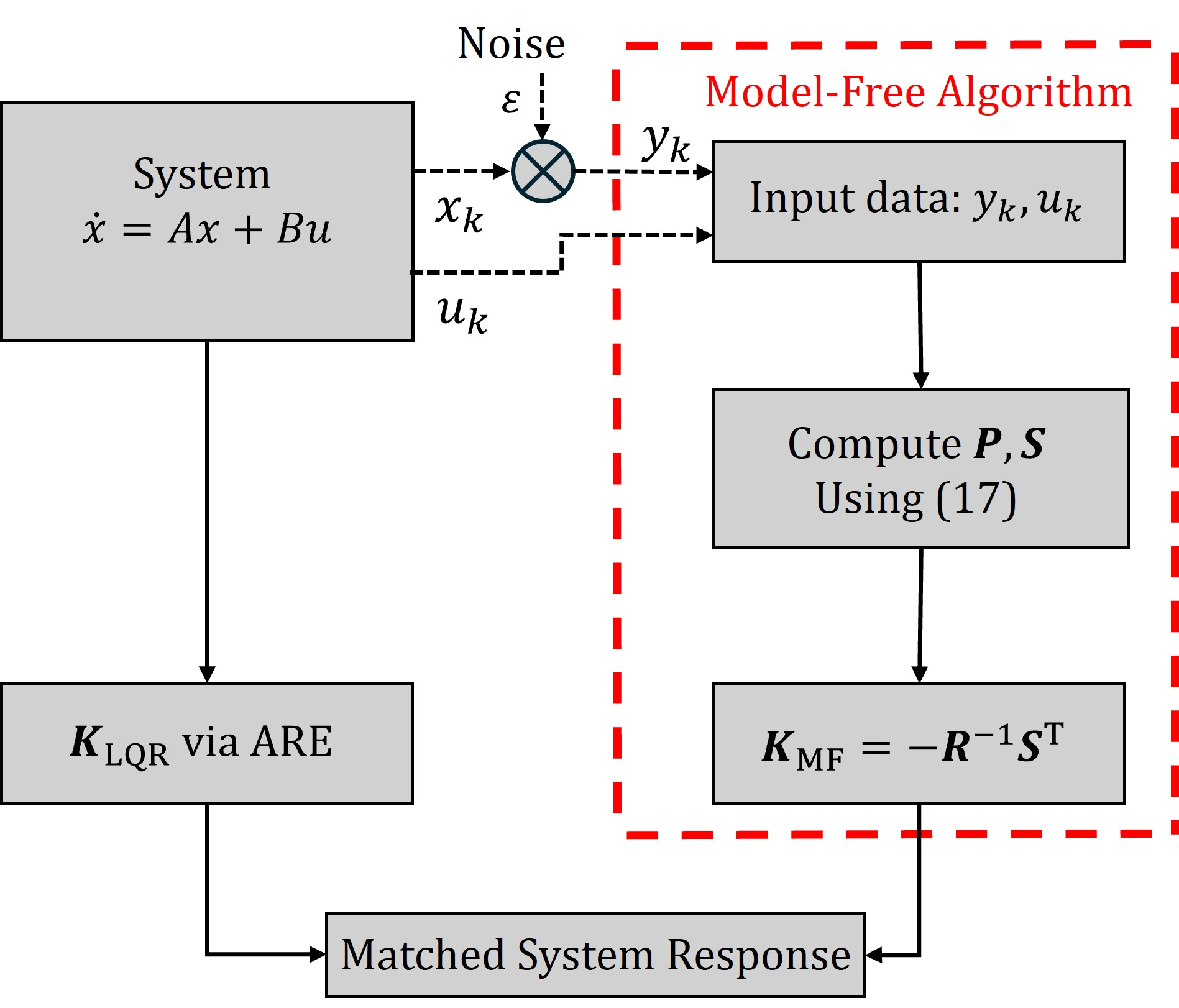}
    \caption{The testing paradigm for Example 1.}
    \label{fig:flowchart}
\end{figure}

\subsection{Known Linear System} \label{sec:results linear}

Our first example is the fourth-order linearized model of the lateral-perturbation equations of motion for an aircraft \cite[Sec. 10.1.2, p. 148]{bryson1994control}. Coefficients for the model are found from \cite[Sec. IX, p. 210]{heffley1972aircraft} and represent a Boeing 747 traveling at 774 ft/s (Mach 0.8) at an altitude of 40,000 ft. The lateral state vector is defined as $x = \begin{bmatrix} \beta & r & p & \phi \end{bmatrix}^\top$, where the state elements are sideslip angle, yaw rate, roll rate, and roll angle, respectively. The input is the rudder deflection, $u = \delta r$, and the model is as follows:
\begin{equation}
\begin{aligned}
        \dot x = & \begin{bmatrix}
            -0.0558 & -0.9968 & 0.0802 & 0.0415 \\
            0.598 & -0.115 & -0.0318 & 0\\
            -3.05 & 0.388 & -0.4650 & 0\\
            0 & 0.0805 & 1 & 0\end{bmatrix}x \\ & + \begin{bmatrix} 0.00729 \\ -0.475 \\ 0.153 \\ 0 \end{bmatrix}u.
        \label{eq:747latsys}
\end{aligned}
\end{equation}
 The data for this example was generated by choosing for the input sequence a linear-frequency chirp, defined as
 $u_k= \psi \sin \lp 2\pi\lp \frac{c}{2} t_k^2 + f_0t_k \rp \rp$, where $\psi =1\mathrm{e}-4$, $f_0 = 1\mathrm{e}-4$, $f_1 = 7\mathrm{e}-2$, and $c = \lp f_1 - f_0 \rp / T$.

The data collected was sampled at 10 Hz for $T=30$ seconds. The generation of $y_k$ included a small amount of standard normal Gaussian random noise added to the signal to mimic sensor noise. The pair $\lp y_k, u_k \rp$ along with user defined LQR weighting matrices, $M = \text{diag}[\begin{matrix}10&1&1&10\end{matrix}]$ and $R = I$, are the inputs to the optimization problem \eqref{eq:NLP}. The optimal closed-loop controller matrix was computed as
\begin{equation}\label{eq:747Kopt}
    K_\text{MF} = \begin{bmatrix} 9.2236 & -6.6657 & -3.1473 & -2.9555\end{bmatrix}.
\end{equation}

Since the original linear system is known, the resulting LQR controller matrix, for the same $M$ and $R$, can be computed directly using MATLAB's LQR solver, which gives
\begin{equation}\label{eq:747Klqr}
    K_\text{LQR} = \begin{bmatrix} 9.3477 & -7.4703 & -3.3029 & -2.9165\end{bmatrix}.
\end{equation}

Comparing \eqref{eq:747Kopt} and \eqref{eq:747Klqr} shows nearly identical resulting controllers. Some small deviation from the ground-truth optimal controller is expected due to the presence of noise in the data, $y_k$, as well as the fact a finite difference approximation was used in \eqref{eq:finte diff in V}.

The closed-loop response of the system with each controller was simulated. Figure \ref{fig:747_roll_tracking} illustrates the system response to a $\pm 10^\circ$ roll doublet input. The command tracking performance of the model-free controller closely resembles the performance of the ground-truth LQR controller.

\begin{figure}
    \centering
    \includegraphics[width=0.9\linewidth]{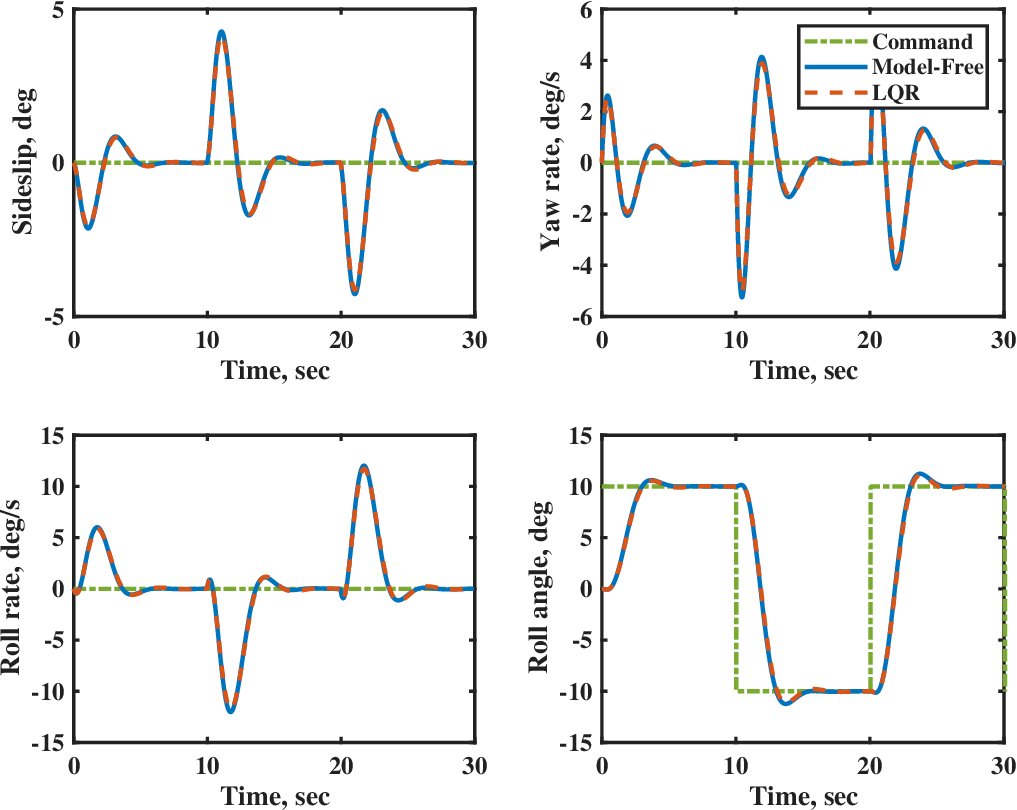}
    \caption{The state response after implementation of a B-747 roll command tracking. The ground truth is generated with $K_{LQR}$ and is shown with the dashed red line. The response when using the controller, $K_{MF}$, computed with the proposed model-free method is shown in blue. Best viewed in color.}
    \label{fig:747_roll_tracking}
\end{figure}

\subsection{Unknown Nonlinear System} \label{sec:results nonlinear}

As a second example, we select a problem of controlling a real-world flight vehicle by utilizing the presented model-free control approach. This process can be used in conjunction with a flight test program to refine vehicle flight performance characteristics that would be unknown or uncertain in the initial design phases.

A high-fidelity nonlinear model of an Holybro X500 V2 quadcopter was implemented in Simulink featuring six-degree-of-freedom (6DOF) equations of motion, electronic speed control (ESC) dynamics, real-world motor and propeller aerodynamic data \cite{tyto1580}, and a FlightGear graphical interface with a joystick input for manual pilot commands. Using a high-fidelity simulation to generate the data was chosen so as to compare the resulting controller to the ground-truth.

The X500-V2 is an X-configuration quadcopter whose motors are labeled 1-4 beginning with the front left motor and counting clockwise. A diagram of the airframe can be seen in Figure \ref{fig:Quad}.

\begin{figure}
    \centering
    \includegraphics[width=0.9\linewidth]{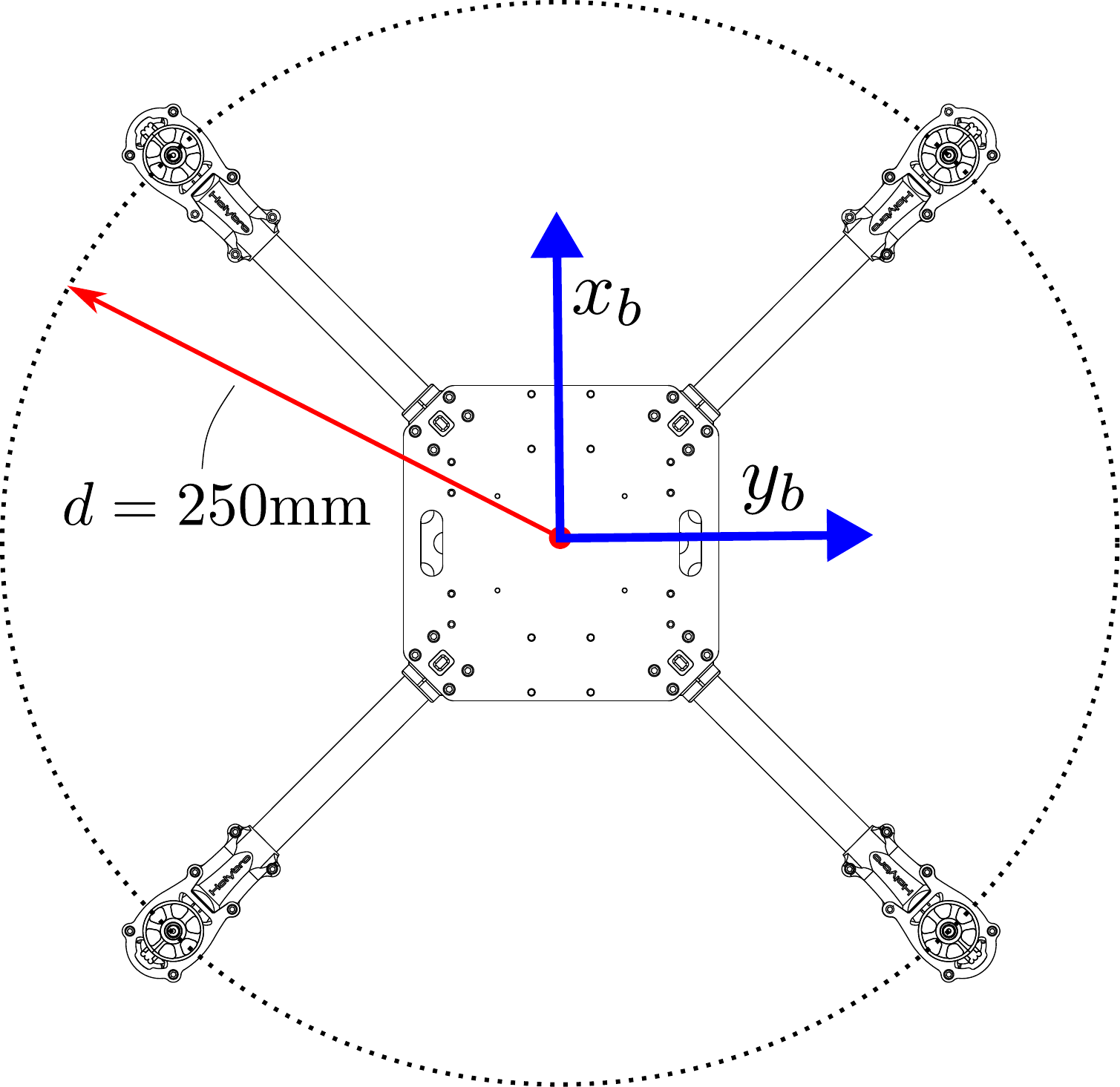}
    \caption{The depiction of the body coordinate frame and dimensions of the Holybro X500 V2 quadcopter for which the Simulink model used for example 2 in Section \ref{sec:results nonlinear} is based on.}
    \label{fig:Quad}
\end{figure}

This simulation uses a Flat-Earth inertial reference frame, $e$, whose principal axis extend in the north-east-down (NED) directions. Additionally, the body-fixed axis, $b$, is defined with the \textit{x}-direction aligned with a vector pointing forward between motors 1 and 2, \textit{y}-direction pointing to the right of the frame between motors 2 and 3, and \textit{z}-direction pointing downward, completing the triad whose origin is at the center of gravity of the body. The twelve states tracked in the simulation are the inertial NED position, $P_e = [\begin{matrix} p_n & p_e & p_d\end{matrix}]^\top$, Euler orientation angles, $\Phi = [\begin{matrix} \phi & \theta & \psi \end{matrix}]^\top$, body-fixed velocity, $V_b = [\begin{matrix} u_b & v_b & w_b\end{matrix}]^\top$, and angular velocity of body $b$ with respect to the inertial frame $e$, $\omega_{b/e} = [\begin{matrix} p & q & r \end{matrix}]^\top$. The translational equations of motion of the quadcopter are given by
\begin{equation}
\begin{cases}
            \begin{aligned}
                \dot P_e & = C_{b/e}(\Phi) V_b\\
                \dot V_b & = \frac{1}{m}F_b - \omega_{b/e} \times V_b,
            \end{aligned}
        \label{eq:6dof trans}
    \end{cases}
\end{equation}
where $C_{b/e}(\Phi)$ is the 3-2-1 Euler rotation matrix from $b$ to $e$. The total external force vector can separated into individual gravity, drag, and propulsive force contributions as
\begin{equation}
\begin{aligned}
        F_b & =  F_{gravity} + F_{drag} + F_{propulsion}.\\
        \end{aligned}
        \end{equation}
 For the quadcopter, these components can be expressed as
\begin{equation}
    \begin{aligned}
        F_b &  = C_{e/b}(\Phi)\begin{bmatrix}
        0\\ 0\\ mg
    \end{bmatrix} + \begin{bmatrix}
        F_{drag,x}\\
        F_{drag,y}\\
        F_{drag,z}\\
    \end{bmatrix} + \begin{bmatrix}
        0\\ 0\\ -T
    \end{bmatrix}.
    \end{aligned}
\label{eq:F_b}
\end{equation}
Each of the drag force components can be computed using information from the drag coefficient $C_d$, a reference area, $S$, and the air density, $\rho$. For example, in the x-direction we have, $F_{drag,x} = - \text{sign}(u_b) \times 1/2\rho u_b^2 \times S_x \times C_{d,x}$. The rotational equations of motion of the quadcopter are given by
\begin{equation}
    \begin{cases}
        \begin{aligned}
            \dot \Phi & = W\lp \Phi\rp \omega_{b/e}\\
            \dot \omega_{b/e} & = I_b^{-1} \left[ M_b - \omega_{b/e} \times \lp I_b \omega_{b/e} \rp \right],
        \end{aligned}
    \end{cases}
    \label{eq:6dof rot}
\end{equation}
where
\begin{equation}
    \begin{aligned}
        W (\Phi) = \begin{bmatrix}
                1 & \sin\phi\tan\theta & \cos\phi\tan\theta\\
                0 & \cos\phi & -\sin\phi\\
                0 & \sin\phi/\cos\theta & \cos\phi/\cos\theta
            \end{bmatrix} 
    \end{aligned}
\end{equation}
is the rotational kinematics matrix and $I_b$ is the moment of inertia matrix for the quadcopter. The applied moment vector can be expanded as
\begin{equation}
    \begin{aligned}
        M_b = \begin{bmatrix}
            M_x \\ M_y \\ M_z
        \end{bmatrix}
    \end{aligned}.
\end{equation}

The simulation was designed to be operated from user-inputs through a joystick yielding a  four dimensional control vector consisting of one throttle and three torque control inputs. The user commands a primary throttle setting, $\Gamma\in[0, 1]$, via a throttle lever, that corresponds to the desired altitude. Rotating the joystick produces three input signals, $\begin{bmatrix}\tau_\text{r} & \tau_\text{p} & \tau_\text{y} \end{bmatrix}^\top \in [-1, 1]$, that perturb the attitude of the quadcopter by specifically increasing or decreasing individual motor speeds to produce the desired body torque. In matrix form the mapping from joystick inputs to commanded motors speeds is
\begin{equation}
    \begin{aligned}
        \begin{bmatrix}
            w_{1,\text{cmd}} \\ w_{2,\text{cmd}}  \\ w_{3,\text{cmd}}  \\w_{4,\text{cmd}} 
        \end{bmatrix} = \begin{bmatrix}
            K_{\Gamma} & -K_{\text{r}} & K_{\text{p}} & K_{\text{y}} \\
            K_{\Gamma} &   K_{\text{r}} & K_{\text{p}} & -K_{\text{y}} \\
            K_{\Gamma} &   K_{\text{r}} & -K_{\text{p}} & K_{\text{y}} \\
            K_{\Gamma} & -K_{\text{r}} & -K_{\text{p}} & -K_{\text{y}} \\
        \end{bmatrix}
        \begin{bmatrix}
            \Gamma \\ \tau_{\text{r}} \\ \tau_{\text{p}} \\ \tau_{\text{y}}
        \end{bmatrix},
        \end{aligned}
\end{equation}
where a scalar gains $K_{\Gamma,\text{r},\text{p},\text{y}}>0$ are used to transform the joystick input signal to appropriate units. The commanded motor speed vector, $w = \begin{bmatrix} w_{1,\text{cmd}} & w_{2,\text{cmd}}  & w_{3,\text{cmd}}  & w_{4,\text{cmd}} \end{bmatrix}^\top $, forms the plant input signals.

Inside the plant model, electronic speed controller dynamics are modeled  to track the relationship between commanded motor speeds, $w_{i,\text{cmd}}$, to achieved motor speed, $w_{i,\text{ach}}$.  The total thrust $T = \sum_{i=1}^4 T_i$ is the sum of all individual motor thrusts, $T_i$, and the applied torques can be computed from motor thrusts as
\begin{equation}
    \begin{aligned}
        \begin{bmatrix}
            M_x \\ M_y
        \end{bmatrix} & = 
        \begin{bmatrix}
            -d\sqrt{2}/2 & d\sqrt{2}/2& d\sqrt{2}/2 & -d\sqrt{2}/2 \\ 
             d\sqrt{2}/2 & d\sqrt{2}/2& -d\sqrt{2}/2 & -d\sqrt{2}/2 \\ 
        \end{bmatrix}
        \begin{bmatrix}
            T_1\\ T_2\\ T_3\\ T_4
        \end{bmatrix}\\
        M_z & = M_1 - M_2 + M_3 - M_4.
    \end{aligned}
    \label{eq:control torques}
\end{equation}
Motor thrusts and moments, $T_i$ and $M_i$, were extracted from a lookup table of experimental data relating commanded motor speeds and the resulting thrust or torque value. The complete set of nonlinear dynamics, $\dot X = f(X,U)$, for the 6DOF model is given with \eqref{eq:6dof trans} and \eqref{eq:6dof rot} where the twelve dimensional vehicle state becomes $X = [\begin{matrix} P_e & V_b & \Phi & \omega_{b/e}\end{matrix}]^\top$. Interested readers should see \cite{stevens_aircraft_2015} for additional 6DOF modeling details. A simplified moment of inertia is calculated from $I^b = \text{diag}\begin{bmatrix}I_{xx} & I_{yy} & I_{zz}\end{bmatrix}$. Physical parameters for the X500 V2 are listed in Table \ref{tab:quad params}.

\begin{table}
\renewcommand{\arraystretch}{1.3}
\caption{\bf X500 V2 Quadcopter Parameters}
\label{tab:quad params}
\centering
\begin{tabular}{|c|c|c|}
\hline
\bfseries Parameter & \bfseries Value & \bfseries Units\\
\hline\hline
Mass, $m$              & 1.3269 & $kg$\\
X-inertia, $I_{xx}$& $0.01295$ & $kg\cdot m^2$\\
Y-inertia, $I_{yy}$& $0.01244$ & $kg\cdot m^2$\\
Z-inertia, $I_{zz}$& $0.01571$ & $kg\cdot m^2$\\
Arm Length, $d$        & 0.25   & $m$\\
Thrust Coefficient, $C_T$ & 3.1539e--5& $N\cdot s^2/rad^2$\\
Torque Coefficient, $C_Q$ & 4.3543e--9&  $N\cdot s^2/rad^2$\\
Min Motor Speed           & 115 & $rad/s$\\
Max Motor Speed           & 907 & $rad/s$\\
\hline
\end{tabular}
\end{table}
The goal is to compute an attitude controller to stabilize the roll and pitch dynamics. First, in-flight state data was recorded by using a pilot-operated joystick to input roll and pitch oscillations into the quadcopter dynamics. Figure \ref{fig:quad_pitchroll_inputs} presents the roll and pitch joystick inputs. These inputs are kept small in magnitude $\lp \pm 1^\circ \rp$ to maintain trajectories near the equilibrium point. Since the system is nonlinear, we expect the system to behave only approximately linearly in a neighborhood around the equilibrium point. By keeping the input small, we avoid a flight region where the response begins to deviate significantly from the approximately linear regime.

Note that while the control inputs in this example were generated manually from a human-operated joystick in an open-loop system, data could also be collected from vehicle using a sub-optimal closed-loop controller.
\begin{figure}
    \centering
    \includegraphics[width=0.9\linewidth]{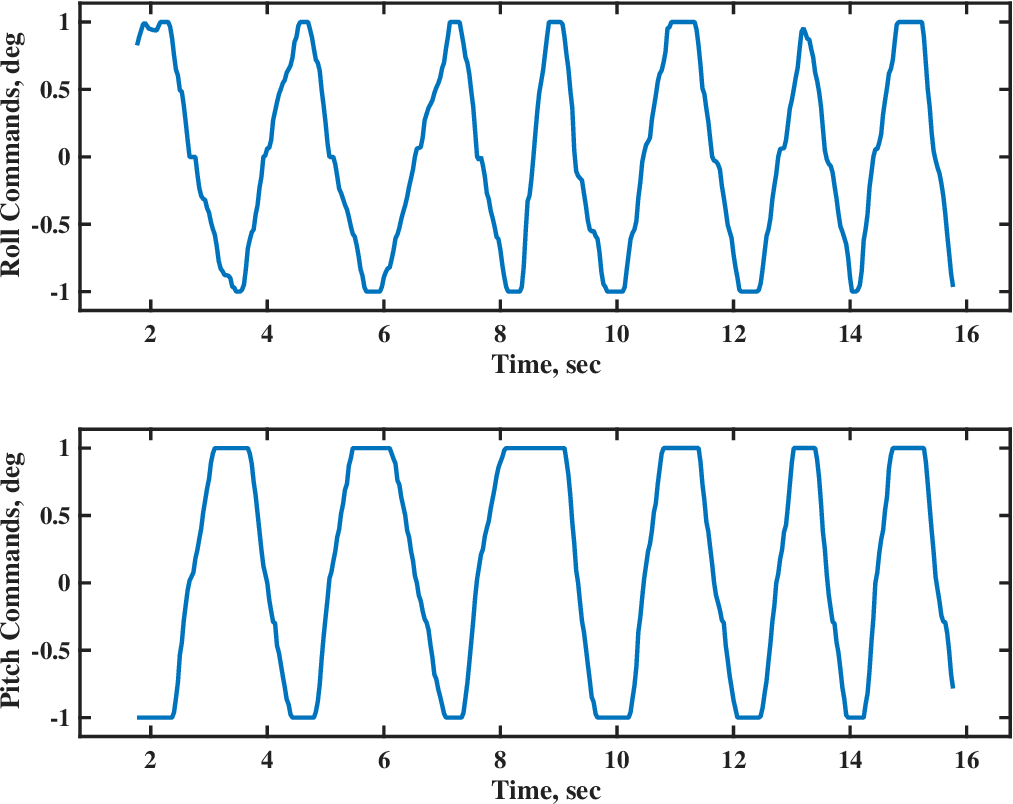}
    \caption{Joystick generated manual control inputs, in the pitch(below) and roll (above) axes. }
    \label{fig:quad_pitchroll_inputs}
\end{figure}

Next, using the generated noisy state and input pair $\lp y_k, u_k \rp$, an optimal controller was computed for the weighting matrices $M = \text{diag}[\begin{matrix}100&100&1&1\end{matrix}]$ and $R = \text{diag}[\begin{matrix}0.1&0.1\end{matrix}]$. We reduce the dimensionality of the state space to only the attitude-related states and inputs for controller synthesis, that is, $x = \begin{bmatrix} \phi & \theta & p & q\end{bmatrix}^\top$ and $u = \begin{bmatrix} \tau_\text{roll} & \tau_\text{pitch}\end{bmatrix}^\top$. The optimal model-free closed loop controller was computed as
\begin{equation}\label{eq:quadKopt}
    K_\text{MF} = \begin{bmatrix} 31.8563 & 0 & 10.2369 & 0 \\
                                   0 & 31.5142 & 0 & 10.2665 \end{bmatrix}.
\end{equation}
The nonlinear quadcopter simulation was linearized about the hover trim condition, a linear system was extracted, and an LQR controller was computed. The LQR gain matrix was calculated as
\begin{equation}\label{eq:quadKlqr}
    K_\text{LQR} = \begin{bmatrix} 31.6228 & 0 & 10.2900 & 0 \\
                                   0 & 31.6228 & 0 & 10.0965 \end{bmatrix}.
\end{equation}
Again, notice the close resemblance between $K_\text{opt}$ and $K_\text{lqr}$, indicating that the model-free controller is approaching the optimal LQR solution.

The performance of each of the controllers was compared by simulating the nonlinear quadcopter closed-loop system. Figure \ref{fig:quad_pitchroll_tracking} shows the quadcopter response to simultaneous roll ($\pm 5^\circ$ at $1/2$Hz) and pitch ($\pm 12^\circ$ at $1/5$Hz) doublets. The command tracking performance of the model-free controller closely resembles the performance of the ground-truth LQR controller.
\begin{figure}
    \centering
    \includegraphics[width=0.9\linewidth]{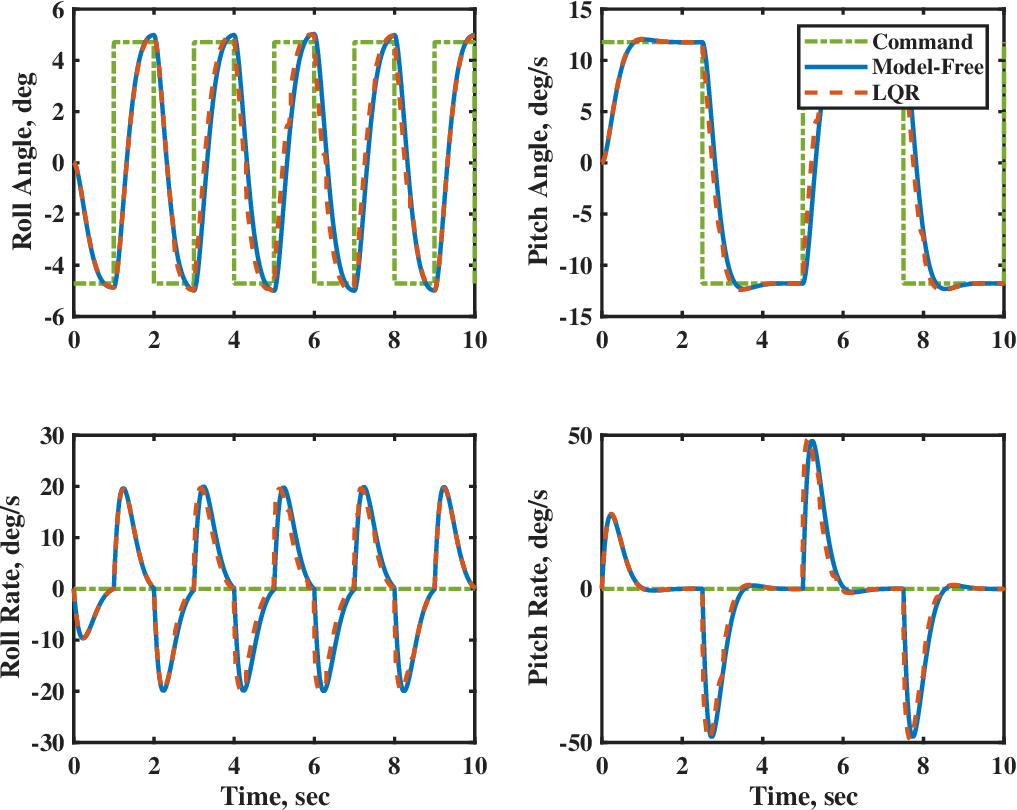}
    \caption{The state response after implementation of an attitude command tracking on a quadcopter 
    drone. The ground truth is generated with $K_{LQR}$ and is shown with the dashed red line. The response when using the controller, $K_{MF}$, computed with the proposed model-free method is shown in blue. Best viewed in color.}
    \label{fig:quad_pitchroll_tracking}
\end{figure}

\section{Conclusion}\label{sec:conclusion}
A model-free approach to infinite horizon LQR controller synthesis, using only data, is presented. Necessary conditions for optimality are derived, using an implicit model-free formulation. This was found by tracking the dynamics of the value function along a trajectory, hence eliminating the dependence on the system matrix. An NLP is constructed which receives a sequence of noisy state measurements, control inputs, and user-defined LQR weighting matrices to compute the optimal feedback control matrix subject to the value function optimality constraint.  We detail how the proposed model-free approach is equivalent, in continuous time, to a specific type of Q-Learning method. Lastly, two examples are presented that indicate that the performance of the model-free controller tracks performance of classical LQR controllers with the benefit of not requiring knowledge of the system matrix.

Future avenues of study include the development of metrics for data collection that ensure a stable estimation of the controller. This would allow users to reduce the total amount of data required for accurate model-free controller synthesis. Additionally, this control design method is planned to be validated on real-world flight vehicles to improve flight performance on systems which are difficult to model comprehensively in simulation programs.


\appendix{}        

Before providing a proof for Lemma \ref{lem:advantage}, we need the following supporting lemma, which is an adaptation of Lemma 6 in \cite{willems_least_1971}:

\begin{lemma}\label{lemma:2}
    Let $P$ be any symmetric solution of the algebraic Riccati equation \eqref{eq:ARE} for arbitrary $x$ and $u$ that satisfy \eqref{eq:LTIsys}. Then on the time interval $\tau\in\left[t,t+T\right]:$
\begin{align}
 & V\left(x\left(t\right)\right)\nonumber \\
 & +\int_{t}^{t+T}\left\Vert u\left(\tau\right)+R^{-1}B^{\top}Px\left(\tau\right)\right\Vert _{R}^{2}d\tau\nonumber \\
 & =V\left(x\left(t+T\right)\right)\nonumber \\
 & +\int_{t}^{t+T}x\left(\tau\right)^{\top}Mx\left(\tau\right)+u\left(\tau\right)^{\top}Ru\left(\tau\right)d\tau.\label{eq:lem for A}
\end{align}
\end{lemma}
\begin{proof}
    Differentiation by parts of \eqref{eq:optimal value funct} yields
    \begin{equation}
        \begin{aligned}
            V(x(t+T)) - & V(x(t)) = \int_t^{t+T} \frac{d}{d\tau} \lp x(\tau)^\top Px(\tau)\rp d\tau \\
            & = \int_t^{t+T} 2x(\tau)^\top P \lp Ax(\tau) + Bu(\tau)\rp d\tau. 
        \end{aligned}
        \label{eq:proof2e1}
    \end{equation}
    The first equality follows from the fundamental theorem of calculus, and the second equality follows from \eqref{eq:LTIsys}. Applying the Riccati equation \eqref{eq:ARE}, the integrand becomes
    \begin{equation}
        \begin{aligned}
             2x^\top& P (Ax + Bu) \\
             & = x^\top (A^\top P + PA)x + u^\top B^\top Px + x\top PBu\\
             & = x^\top (PBR^{-1}B^\top P - M)x + u^\top B^\top Px + x^\top PBu\\
             & = \|u + R^{-1}B^\top Px\|^2_R - u^{\top}Ru - x^\top Mx.
        \end{aligned}
        \label{eq:proof2e2}
    \end{equation}
    Evaluating \eqref{eq:proof2e2} at $x(\tau)$ and $u(\tau)$  and substituting into the integrand of \eqref{eq:proof2e1} produces the correct result.
\end{proof}

We now give the proof of Lemma \ref{lem:advantage}.

\begin{proof}
For simplicity we define the quantity
\[
W\left(\tau\right):=\left\Vert u\left(\tau\right)+R^{-1}B^{\top}Px\left(\tau\right)\right\Vert _{R}^{2}.
\]
Using this notation, \eqref{eq:Advantage formula} can be written as
\begin{align*}
{\mathbb{A}}\left(t;x \lp \cdot \rp,u\left(\cdot\right)\right) & :=\int_{t}^{\infty}W\left(\tau\right)d\tau\\
 & =\int_{t}^{t+T}W\left(\tau\right)d\tau+\int_{t+T}^{\infty}W\left(\tau\right)d\tau.
\end{align*}
This implies that
\begin{equation}
{\mathbb{A}}\left(t;x \lp \cdot \rp,u\left(\cdot\right)\right)=\int_{t}^{t+T}W\left(\tau\right)d\tau+{\mathbb{A}}\left(t+T;x \lp \cdot \rp,u\left(\cdot\right)\right).\label{eq:Advantage proof}
\end{equation}
By the definition given in \eqref{eq:Advantage def},
\[
{\mathbb{A}}\left(t;x,u\left(\cdot\right)\right):=Q\left(t,x\left(t\right),u\left(\cdot\right)\right)-V\left(x\left(t\right)\right).
\]
This implies that we can write \eqref{eq:Advantage proof} as
\begin{align*}
Q\left(t;x\left(\cdot\right),u\left(\cdot\right)\right)-V\left(x\left(t\right)\right)= & \int_{t}^{t+T}W\left(\tau\right)d\tau\\
 & +Q\left(t+T;x\left(\cdot\right),u\left(\cdot\right)\right)\\
 & -V\left(x\left(t+T\right)\right).
\end{align*}
Rearranging the above equation and substituting from Lemma \ref{lemma:2} yields
\begin{align*}
Q\left(t;x\left(\cdot\right),u\left(\cdot\right)\right) &=  Q\left(t+T;x\left(\cdot\right),u\left(\cdot\right)\right)\\
 & + \int_{t}^{t+T} \Big\{x\left(\tau\right)^{\top}Mx\left(\tau\right)\\
 &\quad+u\left(\tau\right)^{\top}Ru\left(\tau\right) \Big\} d\tau,
\end{align*}
and we have arrived at our result. 
\end{proof}

\acknowledgments
This research was supported by the Office of Naval Research under Grant: N00014-24-1-2322. 

The views expressed in this article are those of the authors and do not necessarily reflect the official policy or position of the Air Force, the Department of Defense or the U.S. Government.

Distribution Statement A. Approved for Public Release; Distribution is unlimited. PR 24-0248.

\bibliographystyle{IEEEtran}
\bibliography{main}

\thebiography


\begin{biographywithpic}
{Sean R. Bowerfind}{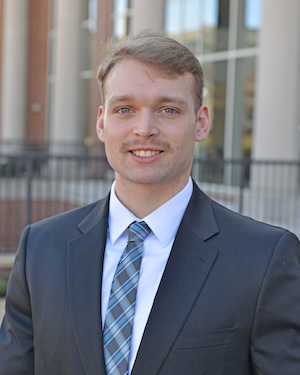}
is a Doctoral student in the Department of Electrical and Computer Engineering at Auburn University. His research areas include flight mechanics, optimal control, trajectory optimization, and applied mathematics. He received his B.S. in Aerospace Engineering from Auburn University in 2022 and his M.S. in Aerospace Engineering from Auburn University in 2023. He was the recipient of Robert G. Pitts Award for the Most Outstanding Senior in Aerospace Engineering and named the Undergraduate Student of the Year by the AIAA Greater Huntsville Section in 2022. Additionally, Sean is a commissioned officer in the United States Air Force, currently assigned to the Air Force Institute of Technology Civilian Institutions Program, and will be attending Euro-NATO Joint-Jet Pilot Training upon completion of his Ph.D. 
\end{biographywithpic} 

\begin{biographywithpic}
{Matthew R. Kirchner}{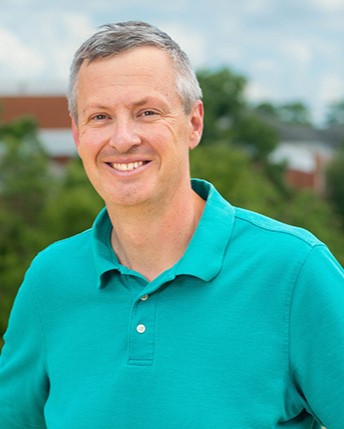}
is an Assistant Professor of Electrical and Computer Engineering at Auburn University. He received his B.S. in Mechanical Engineering from Washington State University in 2007, a M.S. in Electrical Engineering from the University of Colorado at Boulder in 2013, and a Ph.D. in Electrical Engineering from the University of California, Santa Barbara in 2023. He spent over 16 years at the Naval Air Warfare Center Weapons Division, China Lake, first joining the Navigation and Weapons Concepts Develop Branch in 2007 as a staff engineer. In 2012 he transferred into the Physics and Computational Sciences Division in the Research and Intelligence Department, where he served as a senior research scientist. His research interests include level set methods for optimal control, differential games, and reachability; multi-vehicle robotics; nonparametric signal and image processing; and navigation and flight control. He was the recipient of a Naval Air Warfare Center Weapons Division Graduate Academic Fellowship from 2010 to 2012; in 2011 was named a Paul Harris Fellow by Rotary International and in 2021 was awarded a Robertson Fellowship from the University of California in recognition of an outstanding academic record.
\end{biographywithpic}


\begin{biographywithpic}
{Gary A. Hewer}{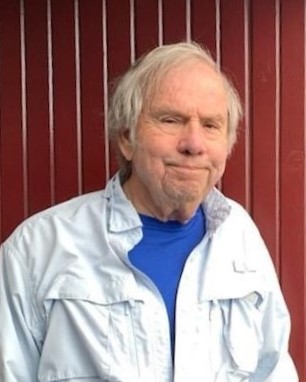}
received his B.A. degree from Yankton College, South Dakota in 1962, and his M.S. degree and Ph.D. degree in Mathematics from Washington State University in 1964 and 1968, respectively. Since 1968, he has worked at the Naval Air Warfare Center Weapons Division, China Lake. He is currently the NAVAIR Senior Scientist for Image and Signal Processing. In 1986 and 1987 he was the interim Scientific Officer for Applied Analysis in the Mathematics Division at the Office of Naval Research (ONR). During his over 50 years of experience he has performed research on control, radar guidance, wavelet applications for both compression and small target detection, image registration and processing, probability theory, and autonomy. In 1987, he received the then Naval Weapons Center’s Technical Director’s Award for his work in control theory and was elected a Senior Fellow of the Naval Weapons Center in 1990 for his work in radar tracking, target modeling, and control theory. In 1998, he
was awarded the Navy Meritorious Civilian Award for his contributions as a Navy research scientist, and in 2002, he was inducted as a NAVAIR fellow. Recently, he was named an Esteemed NAVAIR Fellow.
\end{biographywithpic}

\begin{biographywithpic}
{D. Reed Robinson}{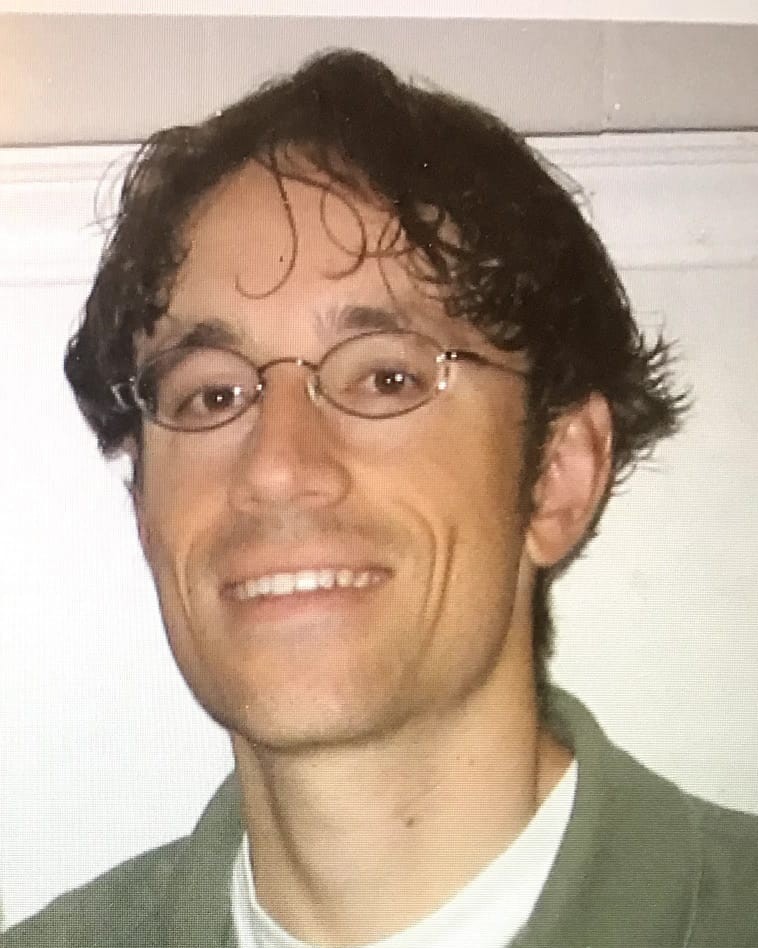}
is a Senior Control Systems Engineer at the Naval Air Warfare Center, Weapons Division since 2003. He received his M.S. in Mechanical Engineering from Utah State University in 2002 with an emphasis in dynamic systems and control. He has developed guidance and autopilot systems for autonomous air and ground vehicles including delivery helicopters and high-speed ground targets. In 2008, he invented spiral-based path-following embedded control algorithms enabling 100+ mph high-speed autonomous ground vehicles (U.S. Patent 8,930,058). More recent work is focused on multi-agent planning algorithms based on model-predictive optimal control and embeddable nonlinear optimization algorithms. He is the co-creator of the NAWCWD Multi-agent Trajectory Planner (MTP), the decABLS  distributed trajectory planning system, and the Bitwise Level Set planner, which uses bitwise operations to enable fast global planning for Dubins dynamics. He was inducted as a NAWCWD Associate Fellow in 2022, and received the McLean Award in 2024 in recognition of inventions important to NAWCWD.
\end{biographywithpic}

\begin{biographywithpic}
{Paula Chen}{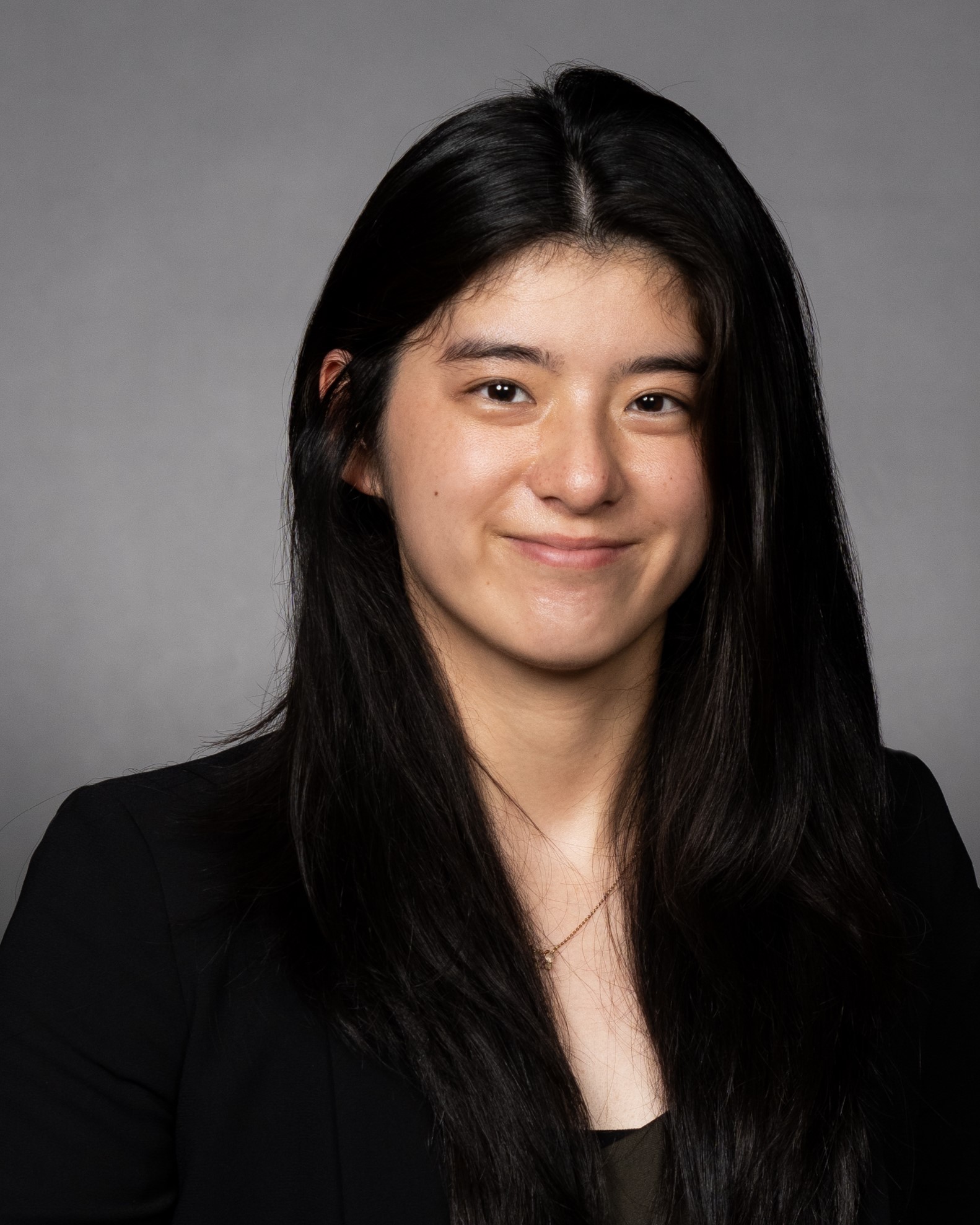}
is a mathematician at the Naval Air Warfare Center Weapons Division, China Lake. She earned her B.A. in mathematics from Dartmouth College in 2017 and her Ph.D. in applied mathematics from Brown University in 2023. Her research interests include Hamilton-Jacobi partial differential equations, optimal control problems, scientific machine learning, efficient algorithms for high-dimensional problems, and FPGA implementations for scientific computing.  
\end{biographywithpic}

\begin{biographywithpic}
{Alireza Farahmandi}{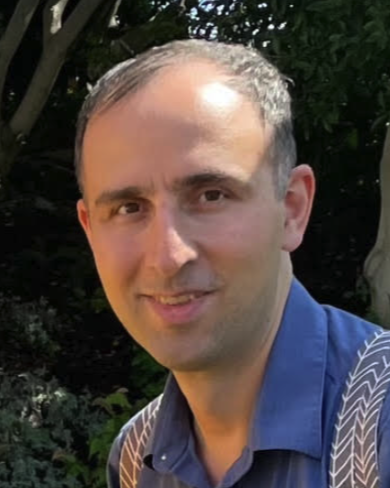}
received his M.S. and Ph.D. in Physics from the University of California, Riverside in 2011 and 2014. He is a research scientist at the Naval Air Warfare Center Weapons Division, China Lake, where he leads projects on intelligent autonomous systems at the Autonomy Research Arena (AuRA). His research interests include adaptive control, autonomous systems and machine learning. He is also an associate faculty member at Barstow Community College where he develops and teaches courses in Astronomy. 
\end{biographywithpic}

\begin{biographywithpic}
{Katia Estabridis}{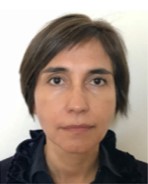}
earned her PhD in Electrical and Computer Engineering from the University of California, Irvine.  She is the Director for the Autonomy Research Arena (AuRA) at the Naval Air Weapons Center Weapons Division (NAWCWD) at China Lake, CA.  She is also the Chief Scientist for Autonomy at the Research Department, NAWCWD.  She leads the Autonomy team on several basic and applied research projects. They include optimal control-based multi-agent trajectory planning algorithms, resource allocation based on multi-objective optimization techniques, reinforcement learning, adaptive control methods and data driven methods for EW applications.  She is a NAWCWD Fellow, holds two patents for facial recognition algorithms and has received the Michelson NAWCWD Award for her Autonomy contributions.  
\end{biographywithpic}

\balance

\end{document}